\documentclass[11pt]{article}
\usepackage{epsfig,url,ulem}
\usepackage{amssymb}
\usepackage{amsmath}
\usepackage{amsthm}
\usepackage{tikz}

\usepackage{graphicx}

\usepackage{setspace}%this stops lists having so much gap in them
\usepackage{enumitem}
\setlist{nolistsep}
\usepackage{lscape}

\newtheorem{theorem}{Theorem}[section]

\newtheorem*{theorem*}{Theorem}

\newtheorem{lemma}[theorem]{Lemma}

\newtheorem{remark}[theorem]{Remark}

\newtheorem{thm}{Theorem}

\newtheorem{cor}[thm]{Corollary}

\newtheorem{example}[theorem]{Example}

\usepackage{fancyhdr}
\begin{document}

\title{A strong Schottky lemma on $n$ generators for $\mathrm{CAT}(0)$ spaces}
\author{Matthew J. Conder and Jeroen Schillewaert}

%\textit{Email addresses:} \url{matthew.conder@auckland.ac.nz},\\ \url{j.schillewaert@auckland.ac.nz}.}
\date{}

\maketitle

\footnotetext{Department of Mathematics, University of Auckland, 38 Princes Street, 1010 Auckland, New Zealand. Both authors are supported by a University of Auckland FRDF grant. The first author is also supported by the Woolf Fisher Trust.}

\begin{abstract}
We give a criterion for a set of $n$ hyperbolic isometries of a $\mathrm{CAT}(0)$ metric space $X$ to generate a free group on $n$ generators.
This extends a result by Alperin, Farb and Noskov who proved this for 2 generators under the additional assumption that $X$ is complete and has no fake zero angles. Moreover, when $X$ is locally compact, the group we obtain is also discrete. We then apply these results to Euclidean buildings.
\end{abstract}

\section{Introduction}

We generalise the main theorem of \cite{AFN} as follows:

\begin{thm}\label{thma}
Let $X$ be a $\mathrm{CAT}(0)$ metric space. Let $g_1, \dots, g_n$ be hyperbolic isometries of $X$ with axes $A_1, \dots A_n$, where $n \ge 2$. Suppose that for each distinct pair of distinct axes $A_i, A_j$ either:
\begin{enumerate}[label=\normalfont{(\Roman*)}]
\item $S_{ij}=A_i \cap A_j$ is a bounded segment, and the two angles $\theta_{ij}^-,\theta_{ij}^+$ between $A_i$ and $A_j$ measured from the two endpoints of $S_{ij}$ are both equal to $\pi$ (as in the left-hand diagram of Figure \ref{axes}); or
\item $A_i$ and $A_j$ are disjoint, and there is a geodesic $B_{ij}$ between $A_i$ and $A_j$ such that all four angles between $B_{ij}$ and $A_i,A_j$ are equal to $\pi$ (as in the right-hand diagram of Figure \ref{axes}).
\end{enumerate}
Additionally, suppose that for each $1 \le i \le n$ there is an open segment $D_i \subseteq A_i$ of length equal to the translation length of $g_i$ such that \vspace{-1mm}
$$ \bigcup_{j \neq i} p_i(A_j)\subseteq D_i, \vspace{-2mm}$$ \label{eqn}
where $p_i \colon X \to A_i$ is the geodesic projection map. Then the subgroup of ${\rm Isom}(X)$ generated by $g_1, \dots, g_n$ is free of rank $n$ and, when $X$ is locally compact, it is also discrete.
\end{thm}

\begin{figure}[h]
\centering
\begin{tikzpicture}
  [scale=0.68,auto=left] 

\node[circle,inner sep=0pt,minimum size=3,fill=black] (1) at (-2,0) {};
\node[circle,inner sep=0pt,minimum size=3,fill=black] (1) at (2,0) {};

\draw (-2,0) to (2,0);

\draw [>-] (-4,2) to (-2,0); \node at (-4.4,2.4) {$A_i$};
 \draw [>-] (-4,-2) to (-2,0); \node at (4.4,2.4) {$A_i$};
 
\draw [->] (2,0) to (4,2);  \node at (-4.4,-2.4) {$A_j$};
\draw [->] (2,0) to (4,-2); \node at (4.4,-2.4) {$A_j$};

\node at (0,0.5) {$S_{ij}=A_i \cap A_j$};

\draw [dashed] (2.5,0.5) to [out=0,in=0, distance=0.4 cm] (2.5, -0.5);
\node at (4,0) {$\theta_{ij}^+=\pi$};
\draw [dashed] (-2.5,0.5) to [out=180,in=180, distance=0.4 cm] (-2.5, -0.5);
\node at (-4,0) {$\pi=\theta_{ij}^-$};

\end{tikzpicture} 
\begin{tikzpicture}
  [scale=0.68,auto=left] 

\node[circle,inner sep=0pt,minimum size=3,fill=black] (1) at (0,2.4) {};
\node[circle,inner sep=0pt,minimum size=3,fill=black] (1) at (0,-2.4) {};

\draw (0,2.4) to (0,-2.4); \node at (0.6,0) {$B_{ij}$};

\draw [>->] (-3,2.4) to (3,2.4);\node at (-3.5,2.4) {$A_i$}; \node at (3.5,2.4) {$A_i$};
\draw [>->] (-3,-2.4) to (3,-2.4);\node at (-3.5,-2.4) {$A_j$}; \node at (3.5,-2.4) {$A_j$};

\draw [dashed] (0.7,2.4) to [out=270,in=0, distance=0.4 cm] (0, 1.7);
\node at (0.8,1.6) {$\pi$};
\draw [dashed] (-0.7,2.4) to [out=270,in=180, distance=0.4 cm] (0, 1.7);
\node at (-0.8,1.6) {$\pi$};

\draw [dashed] (0.7,-2.4) to [out=90,in=0, distance=0.4 cm] (0, -1.7);
\node at (0.8,-1.6) {$\pi$};
\draw [dashed] (-0.7,-2.4) to [out=90,in=180, distance=0.4 cm] (0, -1.7);
\node at (-0.8,-1.6) {$\pi$};

\end{tikzpicture} 

\caption{Cases (I) and (II) of Theorem \ref{thma}.} \label{axes}
\end{figure}
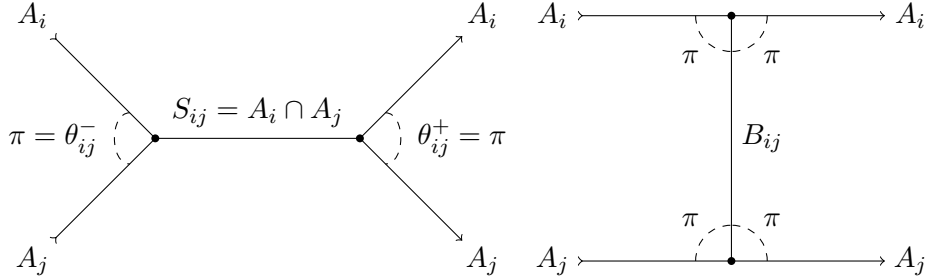

\begin{remark}
By the angle between two geodesic paths, we mean the upper (or Alexandrov) angle, as defined in \textup{\cite[I.1.12]{BH}}.
\end{remark}

\begin{remark}\label{rem:top-equiv}
We only ever consider a topology on ${\rm Isom}(X)$ when $X$ is locally compact. The topology we use is the compact-open topology, which is equivalent to the topology of pointwise convergence in this setting and this gives ${\rm Isom}(X)$ the structure of a topological group which acts continuously on $X$; see \textup{\S 2.4 Theorem 1} and \textup{\S 3.4 Corollary 1} of \textup{\cite[Chapter X]{B}}.
\end{remark}

Theorem \ref{thma} generalises the theorem stated in \cite{AFN} as we allow for an arbitrary finite number of generators and we no longer require that $X$ is complete and has no fake zero angles. Moreover, we also prove discreteness when $X$ is locally compact, and this generalises a result by Lubotzky for isometries of trees \cite[Prop 1.6]{Lub}.

In fact, the main theorem of \cite{AFN} follows directly from Theorem \ref{thma} when $n=2$: in case (I), the projection condition implies that $S_{ij}$ has length strictly less than the translation length of both $g_1$ and $g_2$ and, in case (II), the projection condition always holds since the projection onto each axis is the unique corresponding endpoint of the geodesic $B_{12}$.

\begin{remark}
Karlsson remarked without proof \textup{\cite[end of Section 6]{Ka}} that ``the condition of no-fake angles in \textup{\cite{AFN}}
can be removed and the translation lengths do not necessarily have to
be strictly greater than the length of S". This is part of what we do here, and our proof is similar to the one in \textup{\cite{AFN}}.
\end{remark}

Without the requirement for completeness, Theorem \ref{thma} may be applied to $\mathrm{CAT}(0)$ spaces which are not necessarily complete, such as certain non-discrete Euclidean buildings; see \cite{P} for some background material.

By \cite[4.6.1-4.6.2]{KL} isometries of Euclidean buildings map apartments to apartments, and if the building at infinity is thick, they also map Weyl chambers to Weyl chambers \cite[2.25 and 2.27]{P}. As in \cite{AFN}, we call an isometry $f$ \textit{generic} if none of its parallel axes is contained in any wall of any apartment of $X$. An isometry $f$ is generic if and only if it has a unique invariant apartment $\mathcal{A}_f$ \cite[2.26]{P}. A generic isometry $f$ determines,
for any fixed choice of basepoint $x\in \mathcal{A}_f$, a pair of chambers in $\mathrm{link}(x)$. We say
that generic isometries $f$ and $g$  are \textit{opposite} if $\mathcal{A}_f \cap \mathcal{A}_g = \{x\}$ and each of the
chambers determined by $f$ is opposite in $\mathrm{link}(x)$ to each of the chambers determined
by $g$.

\begin{cor}\label{corc}
Let $X$ be a Euclidean building (where $X^\infty$ is thick) and let $f_1,\cdots f_n$ be hyperbolic isometries of $X$. If $f_1,\cdots f_n$ are pairwise opposite and %their pairwise intersection points 
the pairwise intersection points of their axes
are contained in an open ball of radius at most half the minimum of the translation lengths of $f_1,\cdots f_n$, then $f_1,\cdots f_n$ generate a subgroup of ${\rm Isom}(X)$ which is free of rank $n$. If $X$ is locally compact, then this subgroup is also discrete.
\end{cor}
\begin{proof}
By \cite[Prop 1.12]{P}, two halfrays of $A_i$ and $A_j$ emanating from $x$ are contained in an apartment. Thus the projection of $A_j$ onto $A_i$ is equal to their intersection point. 
By our assumption on the intersection points, and the fact that projection does not increase distances \cite[II.2.4(4)]{BH}, the proof is completed using Theorem \ref{thma}. %and Corollary \ref{corb}.
\end{proof}

Note that the geometric realisation of a simplicial complex (in particular, of a simplicial building) is locally compact if and only if it is locally finite. When $G$ is a linear semi-simple
algebraic group defined over a non-archimedean field $k$, then the Bruhat-Tits
building associated to $G$ \cite{T} is locally compact if and only if $k$ is a local field \cite[p.464]{RTW}. 

\begin{remark}
Although all simplicial buildings have a metrically complete $\mathrm{CAT}(0)$ Davis realisation \textup{\cite[11.1]{D}}, a Euclidean building is not necessarily metrically complete, even if it is a Bruhat-Tits building \textup{\cite{MSSS}}. 
Moreover, the Cauchy completion of a Euclidean building is not necessarily a Euclidean building \textup{\cite[6.9]{Kr}}. One can instead use the theory of ultralimits to embed a Euclidean building into a metrically complete Euclidean building \textup{\cite{KL}}. However, to prove Corollary \ref{corc} we did not need this.

\end{remark}

\section{Proof of Theorem \ref{thma}}

We will use the following statement of the Ping Pong Lemma. This generalises the version in \cite[3.3]{C} to an arbitrary finite number of elements. For the discreteness part, we also remove the condition that the topological group $G$ is metrisable.

\begin{lemma}[The Ping Pong Lemma]\label{PPL}
Let $G$ be a group acting on a set $X$, and let $g_1,\dots,g_n \in G\backslash\{e\}$. Suppose that $X_1^+, X_1^-,\dots, X_n^+, X_n^-$ are non-empty, pairwise disjoint subsets of $X$, which do not cover $X$ and for all $1\le i \le n$ satisfy 
\begin{align*}
g_i(X\backslash X_i^-) \subseteq X_i^+ \;\;\; and \;\;\; g_i^{-1}(X\backslash X_i^+) \subseteq X_i^-.
\end{align*}
Then the subgroup $H=\langle g_1,\dots, g_n \rangle\le G$ is free of rank $n$. 
In the case that $X$ is a topological space and $G$ is a topological group which acts continuously on $X$, if each of the subsets $X_1^+, X_1^-,\dots, X_n^+, X_n^-$ is closed in $X$, then $H$ is also discrete.
\end{lemma}
\begin{proof}
Set $Y=X_1^+ \cup X_1^- \cup \dots \cup X_n^+ \cup X_n^-$ and choose $x \in X \backslash Y$. If $w$ is a non-trivial word in $g_1, \dots, g_n$, then $w(x) \in Y$, therefore $w\neq e$ in $G$. Hence $H$ is free of rank $n$.

For the second part, note that $H$ acts continuously on $X$, that is, the map $H\times X\to X$ is continuous with respect to the product topology. It follows that the inverse image of the open set $X\setminus Y$ is open in $H \times X$. But the intersection of this inverse image with the open set $H\times X\setminus Y$ is $\{e\}\times X\setminus Y$, thus $\{e\}$ is open in $H$ and hence $H$ is discrete.
\end{proof}

\begin{lemma}\label{geodesic-concatenation}
Let $[x,y]$ and $[y,z]$ be geodesics in a $\mathrm{CAT}(0)$ space. If $\angle_y (x,z)=\pi$, then the concatenation $[x,z]=[x,y]\cup [y,z]$ is a geodesic.
\end{lemma}
\begin{proof}
By \textup{\cite[II.1.7(4)]{BH}} the corresponding angle in the relevant comparison triangle is also $\pi$ and thus $d(x,z) = d(x,y)+d(y,z)$.
\end{proof}

\begin{proof}[Proof of Theorem \ref{thma}]
Since geodesics are complete convex subsets in $\mathrm{CAT}(0)$ spaces, the projection maps $p_i$ we use are well-defined \cite[II.2.4]{BH}.

Note that for each $1 \le i \le n$, the open segment $D_i$ is a fundamental domain for the action of $g_i$ on $A_i$. Let $A_i^+$ denote the union of all translates of $\overline{D_i}$ under positive powers of $g_i$. Similarly, let $A_i^-$ denote the union of all translates of $\overline{D_i}$ under negative powers of $g_i$. Then $A_i^+$ and $A_i^-$ are disjoint geodesic rays with $A_i \backslash D_i=A_i^+ \sqcup A_i^-$. Set $X_i^+=p_i^{-1}(A_i^+)$ and $X_i^-=p_i^{-1}(A_i^-)$ for each $i$. We will show that these subsets satisfy the hypotheses of the first part of Lemma \ref{PPL}.

It is straightforward to check that the subsets $X_1^\pm, \dots, X_n^\pm$ are non-empty, closed and that they do not cover $X$. Each $X_i^+$ is also disjoint from $X_i^-$, so to apply Lemma \ref{PPL} we must show that the sets $X_i^\pm$ are disjoint from $X_j^\pm$ for $i \neq j$. Since we can replace $g_i$ and $g_j$ by their inverses, if necessary, it suffices to show that $X_i^+$ and $X_j^+$ are disjoint.

To this end, suppose that $x \in X_i^+ \cap X_j^+$ for some $i \neq j$. Then $p_i(x) \in A_i^+$ and $p_j(x) \in A_j^+$. Note that $p_i(x) \neq p_j(x)$, as otherwise $p_i(x)\in D_i \subseteq A_i \backslash A_i^+$. A similar argument shows that $x \notin A_i \cup A_j$. 

In case (I), let $y_i$ and $y_j$ be the (not necessarily distinct) endpoints of $S_{ij}$ which are closest to $p_i(x)$ and $p_j(x)$ respectively. In case (II), let $A_i\cap B_{ij} = \{y_i\}$ and $A_j\cap B_{ij} = \{y_j\}$. By Lemma \ref{geodesic-concatenation}, the geodesic $[p_i(x),p_j(x)]$ is the concatenation of geodesics $[p_i(x),y_i]\cup [y_i,y_j] \cup [y_j,p_j(x)]$. In particular, $\angle_{p_i(x)} (x,p_j(x))  = \angle_{p_i(x)} (x,y_i) \geq \frac{\pi}{2}$ and  $\angle_{p_j(x)} (x,p_i(x))  = \angle_{p_j(x)} (x,y_j) \geq \frac{\pi}{2}$ by \cite[II.2.4(3)]{BH}. But the triangle with distinct vertices $x,p_i(x),p_j(x)$ has a Euclidean comparison triangle with corresponding angles which are also at least $\frac{\pi}{2}$ by  \cite[II.1.7(4)]{BH}, and this is a contradiction.

It remains to prove that $g_i(X \backslash X_i^-)\subseteq X_i^+$ and $g_i^{-1}(X \backslash X_i^+) \subseteq X_i^-$ for each $1\le i\le n$. As in \cite{AFN}, we first note that $p_i$ commutes with $g_i$. Indeed, for $x \in X$, $p_i(g_i(x))$ is the unique point on $A_i$ which realises the distance $d(g_i(x),A_i)$. It follows that $p_i(g_i(x))=g_i(p_i(x))$, since
$$d(g_i(x),A_i) = d(g_i(x),g_i(A_i)) = d(x,A_i) = d(x,p_i(x)) = d(g_i(x),g_i(p_i(x))).$$
Hence if $x \in X\backslash X_i^-$, then $p_i(g_i(x))=g_i(p_i(x))\in A_i^+$ i.e. $g_i(x) \in X_i^+$. Similarly, $p_i$ commutes with $g_i^{-1}$ and, if $x \in X\backslash X_i^+$, then $p_i(g_i^{-1}(x))=g_i^{-1}(p_i(x))\in A_i^-$ i.e. $g_i^{-1}(x) \in X_i^-$. Thus $g_1, \dots, g_n$ generate a free group of rank $n$ by the first part of Lemma \ref{PPL}.

Finally, we prove discreteness when $X$ is locally compact. The action of ${\rm Isom}(X)$ on $X$ is continuous by Remark \ref{rem:top-equiv} and each of the subsets $X_i^\pm$ is closed in $X$ by \cite[II.2.4(4)]{BH}. Hence the second part of Lemma \ref{PPL} completes the proof of Theorem \ref{thma}.
\end{proof}

\section{Two examples}

There are isometries of locally compact $\mathrm{CAT}(0)$ metric spaces which generate groups which are free but not discrete.

\begin{example}\label{free-not-disc}
\normalfont The matrices $A=\begin{bmatrix}
1 & 2 \\
0 & 1 \end{bmatrix}$ and 
$B=\begin{bmatrix}
1 & 0 \\
2 & 1 \end{bmatrix}$
generate a free group of rank two (the \textit{Sanov subgroup} \cite{Sanov}). However, viewing them as matrices over the $p$-adic numbers $\mathbb{Q}_p$, both $A$ and $B$ are infinite order elliptic isometries of the Bruhat-Tits tree $T_p$ corresponding to $\mathbb{Q}_p$. Hence this subgroup of ${\rm PSL}_2(\mathbb{Q}_p) \le {\rm Isom}(T_p)$ is not discrete.
\end{example}

We conclude with an example satisfying the conditions of Corollary \ref{corc}. 
\begin{example}\label{buildings-ex}
\normalfont 
Let $x$ be a vertex in a locally finite Euclidean building $X$ of type $\tilde{A}_2$.  The link of $x$ is a (not necessarily classical) projective plane $\pi$ of order $n$. 
Chambers in $\pi$ correspond to incident point-line pairs. Two chambers $(p_1,L_1)$ and $(p_2,L_2)$ are opposite in $\pi$ if and only if $p_1\notin L_2$ and $p_2\notin L_1$.

Consider the set of $n+1$ lines $L_1,\dots , L_{n+1}$, which each pass through a fixed point $p$ of $\pi$ and points $p \neq p_i \in L_i$. The corresponding chambers $C_i = (p_i,L_i)$
are pairwise opposite in the link of $x$. For any $1\le k\le \lfloor \frac{n+1}{2} \rfloor$, we can choose $k$ pairs of chambers $P_1,\dots,P_k$. For each pair $P_i = (C,C')$, select interior points $q\in C$ and $q'\in 
C'$ which are opposite in the geometric realisation. By \textup{\cite[1.12]{P}}, the Weyl chambers corresponding to the members of $P_i$ are contained in an unique apartment and the halfrays corresponding to $q$ and $q'$ form a geodesic $A_i$. If these $A_i$ are the axes of hyperbolic isometries $f_i$ then, by construction, these isometries are generic and pairwise opposite and hence, by Corollary \ref{corc}, they generate a subgroup of ${\rm Isom}(X)$ which is both discrete and free of rank $k$.

For example, consider the Bruhat-Tits building $\Delta$ associated to $G=\mathrm{SL}_3(\mathbb{Q}_p)$. If $f$ is a generic isometry with unique invariant apartment $\mathcal{A}_f$ then $sfs^{-1}$ has $s\mathcal{A}_f$ as its unique invariant apartment. Choose Weyl chambers and apartments as above. Since $G$ acts strongly transitively on $\Delta$ \cite{T} it suffices to find one generic isometry $f$ and then the appropriate conjugates of $f$ will generate a free group, which moreover is discrete since $\Delta$ is locally finite. We can take $f$ to be, for instance, the diagonal matrix with entries $1$, $p$ and $p^{-1}$. A more explicit description of the appropriate conjugates of $f$ could then be obtained using similar techniques to \cite{AFN,WZ}.

\end{example}

\flushleft{{\bf Acknowledgement}} We are very grateful for the referee's remarks which improved the paper both stylistically and mathematically.


\begin{thebibliography}{99}
\small
\addtolength{\itemsep}{-0.45\baselineskip}
\bibitem{AFN} R. C. Alperin, B. Farb\ and\ G. A. Noskov, A strong Schottky lemma for nonpositively curved singular spaces, {\it Geom. Dedicata} {\bf 92} (2002), 235--243.

\bibitem{B} N.~Bourbaki, {\it Elements of Mathematics. General Topology. Part 2}, Hermann, Paris, 1966.


\bibitem{BH} M.~Bridson and A.~Haefliger, {\it Metric spaces of non-positive curvature}, Grundlehren der mathematischen Wissenschaften, 319, Springer-Verlag, Berlin, 1999.

\bibitem{C} M.~J.~Conder, Discrete and free two-generated subgroups of $\rm SL_2$ over non-archimedean local fields, {\it J. Algebra} {\bf 553} (2020), 248--267.


\bibitem{D} M.~W.~Davis, Buildings are ${\rm CAT}(0)$, in {\it Geometry and cohomology in group theory (Durham, 1994)}, 108--123, London Math. Soc. Lecture Note Ser., 252, Cambridge Univ. Press, Cambridge.

\bibitem{Ka} A.~Karlsson, On the dynamics of isometries, {\it Geom. Topol.} {\bf 9} (2005), 2359--2394. 

\bibitem{KL} B.~Kleiner and B. Leeb, Rigidity of quasi-isometries for symmetric spaces and Euclidean buildings, {\it Inst. Hautes \'{E}tudes Sci. Publ. Math.} No. 86 (1997), 115--197.

\bibitem{Kr} L.~Kramer, On the local structure and the homology of ${\rm CAT}(\kappa)$ spaces and Euclidean buildings, {\it Adv. Geom.} {\bf 11} (2011), no.~2, 347--369.

\bibitem{Lub} A. Lubotzky, Lattices in rank one Lie groups over local fields, {\it Geom. Funct. Anal.} {\bf 1} (1991), no.~4, 406--431.

\bibitem{MSSS} B.~Martin, J.~Schillewaert, K.~Struyve and G.~Steinke, On metrically complete Bruhat-Tits buildings, {\it Adv. Geom.} {\bf 13} (2013), no.~3, 497--510.

\bibitem{P} A.~Parreau, Immeubles affines: construction par les normes et \'{e}tude des isom\'{e}tries, in {\it Crystallographic groups and their generalizations (Kortrijk, 1999)}, 263--302, Contemp. Math., 262, Amer. Math. Soc., Providence, RI.

\bibitem{RTW} B.~R\'emy, A.~Thuillier, and A.~Werner, Bruhat-Tits theory from Berkovich's point of view. I. Realizations and compactifications of buildings, {\it Ann. Sci. \'{E}c. Norm. Sup\'{e}r. (4)} {\bf 43} (2010), no.~3, 461--554. 

\bibitem{Sanov} I.~N.~Sanov, A property of a representation of a free group (Russian), {\it Doklady Akad. Nauk SSSR (N. S.)} {\bf 57} (1947), 657--659.

\bibitem{T} J.~Tits, Reductive groups over local fields, in {\it Automorphic forms, representations and $L$-functions (Proc. Sympos. Pure Math., Oregon State Univ., Corvallis, Ore., 1977), Part 1}, 29--69, Proc. Sympos. Pure Math., XXXIII, Amer. Math. Soc., Providence, RI.

\bibitem{WZ} S.~Witzel and M.~Zaremsky, A free subgroup in the image of the 4-strand Burau representation, {\it J. Knot Theory Ramifications} {\bf 24} (2015), no.~12, 1550065, 16 pp.

\end{thebibliography}
\end{document}